\documentclass[11pt,a4paper]{amsart}

\title[The polynomially convex embedding dimension]{The polynomially convex embedding dimension of real manifolds of dimension $\leq 11$}

\author{Leandro Arosio$^{\dagger}$ \& H{\aa}kan Samuelsson Kalm \& Erlend F.\ Wold$^{*}$}
\address{Leandro Arosio, Dipartimento Di Matematica, Universit\`{a} di Roma ``Tor Vergata'',
Via Della Ricerca Scientifica 1, 00133 Roma, Italy}
\email{arosio@mat.uniroma2.it}
\address{H{\aa}kan Samuelsson Kalm, Department of Mathematical Sciences, Division of Algebra and Geometry, University of Gothenburg and 
Chalmers University of Technology, SE-412 96 G\"{o}teborg, Sweden}
\email{hasam@chalmers.se}
\address{Erlend F. Wold, Department of Mathematics, University of Oslo, PO-BOX 1053,
Blindern, 0316 Oslo, Norway}
\email{erlendfw@math.uio.no}

\date{\today}

\usepackage[english]{babel}
\usepackage{amsmath,calligra}
\usepackage{xcolor}
\usepackage{amsthm,amssymb,latexsym}
\usepackage[all,cmtip]{xy}
\usepackage{url,ae}
\usepackage{t1enc}
\usepackage{bm}
\usepackage{mathrsfs}
\usepackage{graphicx}
\usepackage{geometry}
\newtheorem{proposition}{Proposition}[section]
\newtheorem{theorem}[proposition]{Theorem}
\newtheorem{lemma}[proposition]{Lemma}
\newtheorem{corollary}[proposition]{Corollary}

\theoremstyle{definition}

\newtheorem{remark}[proposition]{Remark}

\numberwithin{equation}{section}

\DeclareMathOperator{\Hom}{\mathscr{H}\text{\kern -3pt {\calligra\Large om}}\,}
\DeclareMathOperator{\Ext}{\mathscr{E}\text{\kern -3pt {\calligra\Large xt}}\,\,}
\DeclareMathOperator{\Image}{\mathscr{I}\text{\kern -3pt {\calligra\Large m}}\,}
\DeclareMathOperator{\Ker}{\mathscr{K}\text{\kern -3pt {\calligra\Large er}}\,}
\newcommand{\C}{\mathbb{C}}
\newcommand{\debar}{\bar{\partial}}

\newcommand{\Sb}{\mathcal{S}}

\newcommand{\R}{\mathbb{R}}

\newcommand{\PM}{\mathscr{P} \kern -3pt \mathscr{M}}

\newcommand{\CH}{\mathscr{C} \kern -2pt \mathscr{H}}

\def\newop#1{\expandafter\def\csname #1\endcsname{\mathop{\rm #1}\nolimits}}
\newop{span}

\begin{document}
\nocite{*}
\bibliographystyle{plain}

\begin{abstract}
We show that any compact smooth real $n$-dimensional manifold $M$ with $n\leq 11$
can be smoothly embedded into $\C^{n+1}$ as a polynomially convex set. The result is optimal, as there is  no such embedding into $\C^n$. This solves a problem by Izzo and Stout
for $n\leq 11$.
Additionally, we show that 
the image $\widetilde{M}$ of $M$ in $\C^{n+1}$ is stratified totally real. 
As a consequence, by a result in \cite{SW}, each 
continuous complex-valued functions on $\widetilde{M}$ is the uniform limit 
on $\widetilde{M}$ of holomorphic polynomials in $\C^{n+1}$.
Our proof is based on the jet transversality theorem and a recent slight improvement
of a perturbation result by the first and the third author.
\end{abstract}

\thanks{${}^\dagger$  Partially supported by   INdAM, by  PRIN {\sl  Real and Complex Manifolds: Geometry and Holomorphic Dynamics} n. 2022AP8HZ9, and by the MUR Excellence Department Project MatMod@TOV
		CUP:E83C23000330006.}

\thanks{${}^*$ Supported  by the project “Pure Mathematics in Norway", funded by the TMS foundation.}

\maketitle
\thispagestyle{empty}

\section{Introduction}

Let $M$ be a compact real $n$-dimensional $C^\infty$-manifold  without boundary. 
The following natural question was asked by Izzo and Stout 
\cite[Question~4]{IS}.
\emph{What is the smallest integer $N$ such that any $M$ as above can be
smoothly embedded into $\C^N$ as a polynomially convex subset?}
This integer $N$ is the polynomially convex (smooth) 
embedding dimension of $n$-dimensional compact manifolds. 
The main result in this paper is that $N=n+1$ if $n\leq 11$. 
If $n\geq 12$ our techniques give $N\leq \lfloor 5n/4\rfloor-1$.

Before our work, the polynomially convex smooth embedding dimension was known to be
$n+1$ as long as $n\leq 5$. Indeed,
for homological reasons no compact $n$-dimensional topological submanifold of $\C^n$ can be polynomially convex, 
  see e.g. \cite[Corollary 2.3.5]{stout}.
The upper bound $N\leq \lfloor 3n/2\rfloor$ has 
been known since the work by Forstneri{\v c}--Rosay \cite{FR}, and 
Forstneri{\v c} \cite{Fo}; it implies that $N=n+1$ if $n\leq 3$.
Gupta--Shafikov \cite{GS1, GS2} improved the 
Forstneri{\v c}--Rosay bound and showed that $N\leq \lfloor 3n/2\rfloor-1$,
which implies that $N=n+1$ as long as $n\leq 5$.
If the embedding is only required to be topological, then the 
corresponding (topological) embedding dimension is $n+1$, see 
Vodovoz--Zaidenberg, \cite{VZ}.

\smallskip

Recall that a compact set $K\subset\C^m$ is polynomially convex if for each
$z\in\C^m\setminus K$ there is a $p\in\C[z_1,\ldots,z_m]$ such that 
$|p(z)|> \sup_K |p|$.
If $m=1$, then $K$ is polynomially convex if and only if $\C\setminus K$ is
connected, whereas in higher dimensions polynomial convexity is not a topological property.
By the Oka--Weil theorem, any function holomorphic in a neighborhood 
of a polynomially convex compact set $K\subset\C^m$ 
can be approximated uniformly on $K$ arbitrarily well
by holomorphic polynomials. This is one reason for the importance of polynomially 
convex embeddings $M\to\C^m$. 

Recall that a submanifold $\widetilde M\subset\C^m$ is totally real if for each 
$z\in\widetilde M$
the tangent space $T_z\widetilde M$ is totally real, that is, $T_z\widetilde M$
does not contain a complex subspace of positive dimension.
If $\widetilde M$ is totally real and polynomially convex, then by
classical work of H\"{o}rmander--Wermer and Nirenberg--Wells,
$\C[z_1,\ldots,z_{m}]$ is dense in  the uniform algebra $C(\widetilde M,\C)$ of 
continuous complex-valued  functions on $\widetilde M$.

If $f\colon M\to\C^{n+1}$ is an embedding, then $f(M)$ cannot  
be totally real in general; 
the dimension of the ambient space is too small. Let $\text{CRsing}\, f$
be the set of $x\in M$ such that $T_{f(x)}f(M)$ is not totally real.
Here is our main result.

\begin{theorem}\label{main1}
Let $M$ be a compact real $C^\infty$-smooth 
manifold of dimension $n\leq 11$. There is a smooth
embedding
$f\colon M\to \C^{n+1}$ with the following properties.
\begin{itemize}
\item[(a)] $f(M)$ is polynomially convex.
\item[(b)] $\text{CRsing}\, f$ is either empty or a closed real $C^\infty$-smooth submanifold 
of $M$ of codimension $4$ and $f(\text{CRsing}\, f)$
is totally real and polynomially convex. 
\end{itemize}
\end{theorem}

Notice that this theorem implies that $f(\text{CRsing}\, f)\subset f(M)$ is a stratification of 
the polynomially convex set $f(M)$ such that $f(\text{CRsing}\, f)$ and 
$f(M)\setminus f(\text{CRsing}\, f)$ are totally real. This means that 
$f(M)$ is a polynomially convex 
\emph{stratified totally real set} in the sense of \cite{SW}.
It then follows from \cite[Theorem~4.5]{SW} that $\C[z_1,\ldots,z_{n+1}]$ is dense in 
$C(f(M),\C)$. By pre-composing the coordinate functions $z_j$ by $f$ we get

\begin{corollary}\label{main2}
Let $M$ be a compact real $C^\infty$-smooth manifold of dimension $n\leq 11$. 
The uniform algebra $C(M,\C)$ of continuous complex-valued functions on $M$
is generated by $n+1$ smooth functions. 
\end{corollary}


The previous result is optimal,  indeed there cannot be $n$ continuous generators $f_1,\dots,f_n$ of $C(M,\C)$, otherwise the map $z\mapsto (f_1(z),\dots, f_n(z))$  would be  a polynomially convex topological embedding in $\C^n$, which,  as discussed above, cannot exist.
It is an open question whether Corollary~\ref{main2}
holds for all dimensions $n\geq 1$. 
Our method of proof of Theorem~\ref{main1} is based on the jet transversality theorem
and Theorem~\ref{trocadero} below. Theorem~\ref{trocadero} is a slight generalization
of the main result in \cite{AW} and is proved in \cite{ASKW}. 
For general $n$,
with only natural minor modifications, our 
proof of Theorem~\ref{main1} gives

\begin{theorem}\label{main3}
Let $M$ be a compact real $C^\infty$-smooth manifold of dimension $n\geq 12$. If
$m=\lfloor 5n/4\rfloor -1$, then there is a smooth
embedding
$f\colon M\to \C^{m}$ with the following properties.
\begin{itemize}
\item[(a)] $f(M)$ is polynomially convex.
\item[(b)] $\text{CRsing}\, f$ is a closed real $C^\infty$-smooth submanifold of $M$ 
of codimension $2(m-n+1)$
and $f(\text{CRsing}\, f)$
is totally real and polynomially convex. 
\end{itemize}
The uniform algebra $C(M,\C)$ of continuous complex-valued functions on $M$
is generated by $m$ smooth functions.
\end{theorem}

\smallskip
The same result with $m=\lfloor 5n/4\rfloor $ has been  obtained by Gupta and Shafikov \cite{GS3}, see Remark \ref{politics} below.

Let us comment on the restriction $n\leq 11$ in Theorem~\ref{main1}
and put it into context of previous works. 
By transversality considerations, cf.\ \cite{FR, Fo},
if $n\leq 3$, then there is an embedding $f\colon M\to\C^{n+1}$ such that
$\text{CRsing}\, f=\emptyset$. After a slight perturbation of $f$, using well known techniques one can then achieve that $f(M)$ is polynomially convex.
Similar transversality considerations show that
there is an embedding $f\colon M\to\C^{n+1}$ such that
$\text{dim}\, \text{CRsing}\, f=0$ and $\text{dim}\, \text{CRsing}\, f=1$ if
$n=4$ and $n=5$, respectively. In these cases it is not immediate how to perturb $f$
to get $f(M)$ to be polynomially convex. The cases
$\text{dim}\, \text{CRsing}\, f=0$ and $\text{dim}\, \text{CRsing}\, f=1$ are handled by Gupta--Shafikov, \cite{GS1, GS2}. Notice that in these cases, 
$f(\text{CRsing}\, f)$ is automatically totally real.
If $n\geq 6$ and $f\colon M\to\C^{n+1}$ is an embedding, then 
$f(\text{CRsing}\, f)$ may not be totally real. The technical part of this paper is to 
show that one in fact can assume that $f(\text{CRsing}\, f)$ is totally real
if $n\leq 11$, possibly after a small perturbation of $f$.
If $n\geq 12$ and $f\colon M\to\C^{n+1}$ is an embedding, 
then $\text{CRsing}\, f$ is not a manifold in general. It will however 
have a stratification into manifold pieces. We see no immediate conceptual obstruction
to our idea of perturbing $f$ so that the pieces become totally real and then 
use (a suitable version of) Theorem~\ref{trocadero} 
to make a further perturbation of $f$ so that $f(M)$ 
becomes polynomially convex. However, considerable technical problems will arise
and we do not know how to handle them at the moment.


\smallskip

The outline of the paper, as well as of the proof of Theorem~\ref{main1}, is as follows.
After some preliminaries in Section~\ref{prelsec}, 
in Section~\ref{crsingsmooth} we show using the jet transversality theorem
that any smooth embedding $f\colon M\to\C^{n+1}$ can be slightly perturbed to
a smooth embedding such that its set of CR-singular points
becomes  a smooth manifold
of dimension $n-4$ if $n\leq 11$. Most parts of this section are well-known to experts,
but for future reference we supply some details.
In Section~\ref{crsingtotreal} we show that a smooth embedding as above can be 
perturbed so that the set of CR-singular points
becomes  a smooth totally real manifold. To do this we consider a certain subset of
the second jet space $J^2(M,\C^{n+1})$ and use the jet transversality theorem.
In Section~\ref{pfthm} we conclude the proof of Theorem~\ref{main1}
by using Theorem~\ref{trocadero} to perturb $f$ further so that $f(M)$ 
becomes polynomially convex. 
In Section~\ref{extra} we indicate how the proof of Theorem~\ref{main1} 
is modified to a proof of Theorem~\ref{main3}.

\begin{remark}\label{politics}
Some days after the appearence on March 2025 of the first version of this work  on arXiv, containing Theorem \ref{main1} and Corollary \ref{main2},  a related preprint \cite{GS3} 
by Gupta and Shafikov was posted on arXiv, independently establishing 
the bound $N\leq \lfloor 5n/4\rfloor$ on the polynomially convex embedding dimension with different methods.
A second version of our work, containing  Theorem~\ref{main3} with the bound $N\leq \lfloor 5n/4\rfloor-1$ was posted on arXiv on October 2025. The  proof  of Theorem~\ref{main3} is a straightforward generalization of our proof of   Theorem~\ref{main1} using our methods.
\end{remark}

\smallskip

\noindent {\bf Acknowledgment:} We are grateful for 
the hospitality and support of Universit\`{a} di Roma ``Tor Vergata'', 
where parts of this work was done during visits of the second and the third author.

\section{Preliminaries and notation}\label{prelsec}
Let $\mathbb{K}$ be either $\R$ or $\C$ and let $\mathbb{K}^{d\times k}$ be the space of $d\times k$-matrices with
entries in $\mathbb{K}$. If $A\in\mathbb{K}^{d\times k}$ we let 
$$
\text{span}_\mathbb{K}\, A\subset \mathbb{K}^d
$$ 
be the column space of $A$. 
Notice that if $B$ is an invertible $k\times k$-matrix,
then
\begin{equation}\label{ABA}
\text{span}_\mathbb{K}\, AB = \text{span}_\mathbb{K}\, A.
\end{equation}
If $\mathbb{K}=\C$, then
$$
\text{span}_\R\, A\subset \mathbb{C}^d
$$ 
is the $\R$-subspace generated by the columns of $A$ over $\R$.
Coordinate vectors in $\mathbb{K}^d$ are always columns in this paper unless 
explicitly stated otherwise. We let $I_{k}$ be the identity $k\times k$-matrix.

If $V$ is a vector space over $\mathbb{K}$, let $G(k,V)$ be the Grassmannian of $k$-dimensional $\mathbb{K}$-subspaces of $V$. If 
$V$ is a vector space over $\C$, then it is also a vector space over $\R$ and we 
let $G_\mathbb{R}(k,V)$ be the Grassmannian of real $k$-dimensional subspaces
of $V$. 

Suppose that $\text{dim}_{\mathbb{K}}V=d$. If we choose a basis for $V$ we can identify $V$ with $\mathbb{K}^d$ and get 
local coordinates on $G(k,V)$ centered at $P\in G(k,V)$ as follows.
After the identification $V\simeq \mathbb{K}^d$,
take $A\in\mathbb{K}^{d\times k}$ such that
$P=\text{span}_\mathbb{K}\, A$. 
By renumbering the standard basis vectors in
$\mathbb{K}^d$ one can assume that the top $k\times k$-submatrix $B$ of $A$
is invertible. In view of \eqref{ABA} one can thus assume that the top
$k\times k$-submatrix of $A$ is $I_k$. A representation $P=\text{span}_\mathbb{K}\, A$
with $A=[I_k \,\, a]^t$ is unique and it follows that 
\begin{equation}\label{stad}
\mathbb{K}^{(d-k)\times k}\ni x\mapsto 
\text{span}_\mathbb{K}\, 
\begin{bmatrix}
I_k \\
a+x
\end{bmatrix}
=\text{span}_\mathbb{K}\, 
\begin{bmatrix}
I_k & 0 \\
a & I_{d-k}
\end{bmatrix}
\begin{bmatrix}
I_k \\
x
\end{bmatrix}
\end{equation}
is a local chart centered at $P$. 
In particular it follows that 
\begin{equation}\label{halt}
\text{dim}_\mathbb{K}\, G(k,V) = k(d-k).
\end{equation}

\smallskip

If $X$ and $Y$ are $C^\infty$-manifolds, let $J^k(X,Y)$ be the $k$th jet space of
$C^\infty$-mappings $X\to Y$. Recall that the points in $J^k(X,Y)$ are equivalence 
classes of smooth mappings $X\to Y$. If $x\in X$ and $f_1, f_2\colon X\to Y$ are smooth,
then $f_1\sim_{x,k} f_2$ if $f_1(x)=f_2(x)$ and the Taylor expansions of $f_1$ and $f_2$
at $x$ agree to order $k$ (with respect to some, and hence any local coordinates
in $X$ and $Y$ centered at $x$ and $f_j(x)$, respectively).
A point $p\in J^k(X,Y)$  then is $[f]_{x,k}$ for some unique $x\in X$; we let $s(p)=x$ and 
$t(p)=f(x)$ be the source and target, respectively, of $p$. If $k=0$, then $s$ and $t$
are the natural projections from $J^0(X,Y)=X\times Y$ to the first and second factor, respectively. There are natural submersions $J^{k+1}(X,Y)\to J^k(X,Y)$ for $k\geq 0$.

We only need to consider $J^1(X,Y)$ and $J^2(X,Y)$
in this paper.  We have
$$
J^1(X,Y)=\{(x,y,\tau); \, x\in X, y\in Y, \tau\in\text{Hom}(T_xX,T_yY)\}.
$$
Let  $f\colon X\to Y$ be a smooth mapping. Then there is an induced mapping
$$
j^1f\colon X\to J^1(X,Y),\quad x\mapsto j^1f(x):=[f]_{x,1}=(x,f(x),Df_x).
$$
On the other hand, if $p=[g]_{x,1}\in J^1(X,Y)$, then the differential 
$Dg_x$ of $g$ at $x$ is well-defined and we denote it by $D_p$.
Associated with $f$ we also have 
$$
j^2f\colon X\to J^2(X,Y),\quad x\mapsto j^2f(x):=[f]_{x,2}.
$$
If $p=[g]_{x,2}\in J^2(X,Y)$, then the differential of $j^1g$ is well-defined at $x$; 
we denote it by
$H_p$ and notice that it is a mapping 
\begin{equation}\label{Hp}
H_p\colon T_{s(p)}X\to T_{\big(s(p),t(p),D_p\big)}J^1(X,Y).
\end{equation}

\section{Making $\text{CRsing}\, f$ smooth}\label{crsingsmooth}
Let $\Sb_j\subset G_\R(n,\C^{n+1})$ be the set of 
real $n$-planes in $\C^{n+1}$ containing a complex $j$-plane but no complex
$j+1$-plane. Notice that
$$
\Sb:=\Sb_1\cup\Sb_2\cup\cdots\cup\Sb_{\lfloor n/2\rfloor}
$$ 
consists of those real $n$-planes that contain a complex line.

\begin{proposition}\label{antarctica}
The set $\Sb$ is a real-analytic closed connected subset of $G_\R(n,\C^{n+1})$
and $\Sb_j$ is a submanifold of $G_\R(n,\C^{n+1})$ of (real) codimension $2j(1+j)$.
Moreover, 
$\overline{\Sb_j}\setminus \Sb_{j}=\Sb_{j+1}\cup\cdots\cup \Sb_{\lfloor n/2\rfloor}$.
\end{proposition}


\begin{proof}[Sketch of proof]
We first check that $\Sb$ is real-analytic. Let $P\in G_\R(n,\C^{n+1})$ and take
$A^\C\in\C^{(n+1)\times n}$ such that $P=\text{span}_\R A^\C$. 
By the usual identification
$\C^{n+1}\simeq\R^{2n+2}$ we get $A\in\R^{(2n+2)\times n}$
such that $P=\text{span}_\R A$. After renumbering the standard basis vectors in
$\R^{2n+2}$ we can assume that $A=[I_n \,\, a]^t$. By \eqref{stad} we get local
coordinates $x\in\R^{(n+2)\times n}$ for $G_\R(n,\C^{n+1})$ centered at $P$.
In terms of the standard $\C$-basis for $\C^{n+1}$ this gives us 
$A_x^\C\in\C^{(n+1)\times n}$ such that $x\mapsto \text{span}_\R A_x^\C$
is a local chart at $P$ and the entries of $A_x^\C$ are linear expressions in $x$. 
Notice that  $\text{span}_\R A_x^\C$ contains a complex line if and only if the columns of 
$A_x^\C$ are linearly dependent over $\C$; cf., e.g., Lemma~4.1 in \cite{SW}. 
Thus, in the local coordinates $x$, $\Sb$
is the common zero set of all complex $n\times n$-minors of $A_x^\C$. 
Since the entries of $A_x^\C$ are linear in $x$ this clearly is a real-analytic 
subset of $\R^{(n+2)\times n}$. 

\smallskip

Now,
let $\Gamma_j\to G(j,\C^{n+1})$ be the fiber bundle whose fiber over a point in 
$G(j,\C^{n+1})$ corresponding to the complex $j$-plane $\pi\subset\mathbb{C}^{n+1}$
is $G_\R(n-2j,\pi^\perp)$. Clearly, $\Gamma_j$ is a compact smooth connected manifold. Using \eqref{halt} and that $\pi^\perp\simeq\C^{n+1-j}$
one checks that
$$
\text{dim}_\R \,\Gamma_j=n^2+2n-2j(1+j).
$$
There is a natural mapping $g_j\colon \Gamma_j\to G_\R(n,\C^{n+1})$ defined as follows.
If $\pi\in G(j,\C^{n+1})$ and $\Pi$ is a real $n-2j$-plane in $\pi^\perp$, then
$$
g_j(\pi,\Pi)=\pi\oplus_\R\Pi.
$$
This mapping turns out to be smooth. Notice
that the image of $g_j$ is $\Sb_j\cup\cdots\cup \Sb_{\lfloor n/2\rfloor}$,
which thus in particular is a closed connected subset of $G_\R(n,\C^{n+1})$. 

\smallskip

We focus on describing $\Sb_1$; the other $\Sb_j$ can be handled similarly. 
Let $P\in\Sb_1$. We will show that there is a neighborhood basis
 $\mathcal{U}_\nu$ of $P$ in $G_\R(n,\C^{n+1})$ such that the restrictions of $g_1$ to
$g_1^{-1}(\mathcal{U}_\nu)$ are injective immersions and $g_1^{-1}(\mathcal{U}_\nu)$
is a neighborhood basis of $g_1^{-1}(P)$.
Since injective immersions locally are embeddings it follows that $\Sb_1$ is a submanifold.
Its codimension then is
$$
\text{dim}_\R\,  G_\R(n,\C^{n+1}) - \text{dim}_\R\,\Gamma_1
=n(n+2)-(n^2+2n-4)=4.
$$

Since $P\in\Sb_1$ there is a unique complex line  $\ell_0$ in $P$. After a unitary change of 
coordinates in $\C^{n+1}$ we can assume that $\ell_0=\C \varepsilon_0$, where
$\varepsilon_0=(1,0\ldots,0)^t$.
Let $A\in\C^{(n+1)\times (n-2)}$ be such that the
columns of $A$ are orthogonal to $\C \varepsilon_0$ and
$$
P=\C \varepsilon_0
+ 
\text{span}_\R\, A.
$$
Notice that $\text{span}_\R\, A$ is uniquely determined.
Since $P\in\Sb_1$, $\text{span}_\R\, A$ is totally real. The columns 
$\varepsilon_1,\ldots,\varepsilon_{n-2}$ of 
$A$ must therefore be linearly independent over $\C$. 
Choose $\varepsilon_{n-1},\varepsilon_n\in\mathbb{C}^{n+1}$ to be orthogonal to
$\C\varepsilon_0$ and such that
$\varepsilon=\{\varepsilon_0,\varepsilon_1,\ldots,\varepsilon_n\}$ is a $\C$-basis for 
$\C^{n+1}$.
For the rest of this proof sketch all coordinate vectors in $\C^{n+1}$ are with
respect to the $\varepsilon$-basis.

\smallskip

Let $\mathcal{U}$ be a neighborhood of $P$ in $G_\R(n,\C^{n+1})$
such that $\mathcal{U}\cap\Sb=\mathcal{U}\cap\Sb_1$.
Possibly shrinking $\mathcal{U}$ we can assume that
if $P'\in\mathcal{U}\cap\Sb$, then the complex line in $P'$ is 
$\ell_w=\C (1,w)^t$ for a unique $w\in\C^{n}$ with $|w|\ll 1$;
recall that we are using the $\varepsilon$-basis.
Notice that all sufficiently small neighborhoods of $P$ have these properties.

The orthogonal complement of $\ell_w$ is the set of $(z_0,\ldots,z_n)^t\in\C^{n+1}$
such that
$$
0=z_0+z^tH\bar{w},
$$
where $z=(z_1,\ldots,z_n)$ and $H$ is the Hermitian $n\times n$-matrix with
$H_{jk}=\langle\varepsilon_j,\varepsilon_k\rangle$,
$1\leq j,k\leq n$.
The orthogonal complement of $\ell_w$ thus is
\begin{equation}\label{jump}
\text{span}_\C\,
\begin{bmatrix}
-(H\bar{w})^t \\
I_{n}
\end{bmatrix}
=
\text{span}_\R\,
\begin{bmatrix}
-(H\bar{w})^t & -i(H\bar{w})^t \\
I_{n} & iI_{n}
\end{bmatrix}.
\end{equation}
Let $\Pi_w\in G_\R(n-2,\ell_w^\perp)$ be the $\R$-span of the first $n-2$ columns of \eqref{jump}; notice that $P=\ell_0\oplus_\R\Pi_0$.
In view of \eqref{stad} we have local coordinates $\xi\in\R^{(n+2)\times (n-2)}$ in  $G_\R(n-2,\ell_w^\perp)$ centered at $\Pi_w$ such that  
$$
\xi\mapsto
\text{span}_\R\,
\begin{bmatrix}
-(H\bar{w})^t & -i(H\bar{w})^t \\
I_{n} & iI_{n}
\end{bmatrix}
\begin{bmatrix}
I_{n-2} \\
\xi
\end{bmatrix}
$$
is a chart at $\Pi_w$.
Then $(w,\xi)$ are local coordinates in $\Gamma_1$ centered at $(\ell_0,\Pi_0)$. The
mapping $g_1$ expressed in these coordinates is
\begin{equation}\label{snorlax}
g_1(w,\xi)=\text{span}_\R\,
\begin{bmatrix}
\begin{matrix}
1 & i \\
w & iw
\end{matrix} &
\begin{bmatrix}
-(H\bar{w})^t & -i(H\bar{w})^t \\
I_{n} & iI_{n}
\end{bmatrix}
\begin{bmatrix}
I_{n-2} \\
\xi
\end{bmatrix}
\end{bmatrix}.
\end{equation}
By construction, we can assume that $\mathcal{U}$ is such that 
$g_1^{-1}(\mathcal{U})=\{(w,\xi);\, |w|\ll 1, |\xi|\ll 1\}$.
The restriction of $g_1$ to $g_1^{-1}(\mathcal{U})$ thus is injective.
Moreover, it follows that if $\mathcal{U}_\nu\subset\mathcal{U}$ 
is a neighborhood basis of $P$, then $g_1^{-1}(\mathcal{U}_\nu)$
is a neighborhood basis of $(\ell_0,\Pi_0)$. 

We now show that $g_1$ is an immersion at $(\ell_0,\Pi_0)$.
To do this we rewrite \eqref{snorlax} by identifying
coordinate vectors $(z_0,\ldots,z_n)^t\in\C^{n+1}$ with
\begin{equation}\label{snorlaxen}
(\text{Re}\,z_0,\text{Im}\,z_0,\text{Re}\,z_1,\text{Re}\,z_2,\ldots,
\text{Re}\, z_n,\text{Im}\,z_1,\text{Im}\,z_2,\ldots,
\text{Im}\, z_n)^t\in\R^{2n+2}.
\end{equation}
For notational convenience we let
$$
h_{w,\xi}=
\begin{bmatrix}
-(H\bar{w})^t & -i(H\bar{w})^t
\end{bmatrix}
\begin{bmatrix}
I_{n-2} \\
\xi
\end{bmatrix}.
$$
If $w=u+iv$ we now have
\begin{equation}\label{pika}
g_1(u,v,\xi)=\text{span}_\R\,
\begin{bmatrix}
I_2 & \begin{matrix}
         \text{Re}\, h_{w,\xi} \\
         \text{Im}\, h_{w,\xi}
         \end{matrix} \\
\begin{matrix}
u & -v
\end{matrix} &
\begin{matrix}
I_{n-2} \\
\xi'
\end{matrix} \\
\begin{matrix}
v & u
\end{matrix} &
\xi''
\end{bmatrix},
\end{equation}
where $\xi'$ is the $2\times (n-2)$-matrix consisting of the first two rows of $\xi$
and $\xi''$ is the $n\times (n-2)$-matrix consisting of the last $n$ rows of $\xi$.
Let $B=B(u,v,\xi)$ be the top $n\times n$-submatrix of the right-hand side of 
\eqref{pika}. Notice that if $|u|\ll 1$, $|v|\ll 1$, and $|\xi|\ll 1$,
then $B$ is invertible and 
$$
B^{-1}=I_n+O(|u|+|v|+|\xi|).
$$
In view of \eqref{ABA} thus
\begin{eqnarray*}
g_1(u,v,\xi) &=& \text{span}_\R\,
\begin{bmatrix}
I_n \\
\begin{bmatrix}
\begin{matrix}
u'' & -v''  \\
v & u 
\end{matrix} &
\xi
\end{bmatrix} B^{-1}
\end{bmatrix}
B \\
&=&
\text{span}_\R\,
\begin{bmatrix}
I_n \\
\begin{matrix}
u'' & -v'' \\
v & u 
\end{matrix}
\quad\xi+O\big((|u|+|v|+|\xi|)^2\big)
\end{bmatrix},
\end{eqnarray*}
where $u''=(u_{n-1},u_n)^t$ and similarly for $v''$.
The lower $(n+2)\times n$-submatrix of the right-hand side is $g_1(u,v,\xi)$ expressed
in a chart of $G_\R(n,\C^{n+1})$ centered at $\ell_0\oplus\Pi_0$; cf.\ \eqref{stad}.
Expressed in this chart, the differential of $g_1$ at $(\ell_0,\Pi_0)$ is the linear mapping
\begin{equation}\label{haalslag}
\R^n\times\R^n\times\R^{(n+2)\times(n-2)}\ni
(u,v,\xi)\mapsto
\begin{bmatrix}
\begin{matrix}
u'' & -v'' \\
v & u
\end{matrix} &
\xi
\end{bmatrix},
\end{equation}
which clearly is injective. Thus $g_1$ is an immersion at $(\ell_0,\Pi_0)$.
Since $(\ell_0,\Pi_0)$ is an arbitrary point in $g_1^{-1}(\Sb_1)$
it follows that $g_1$ is an immersion at every point in $g_1^{-1}(\Sb_1)$, in particular in
$\mathcal{U}$. We have thus showed that the restriction of $g_1$ to 
$g_1^{-1}(\mathcal{U})$ is an injective immersion.
\end{proof}


Consider the dense open subset
$$
\widetilde{J^1}(M,\C^{n+1}):=\{(x,z,\tau)\in J^1(M,\C^{n+1});\, 
\text{rank}_\R\,\tau =n\}
$$
of $J^1(M,\C^{n+1})$. We have a natural mapping
\begin{equation}\label{pinatrium}
\pi_1\colon \widetilde{J^1}(M,\C^{n+1})\to G_\R(n,\C^{n+1}),\quad
(x,z,\tau)\mapsto \tau(T_xM),
\end{equation}
 and we let
$$
\Sb':=\pi_1^{-1} (\Sb), \quad
\Sb'_j:=\pi_1^{-1} (\Sb_j).
$$

Let $f\colon M\to\C^{n+1}$ be an embedding and recall that there is an induced mapping $j^1f\colon M\to J^1(M,\C^{n+1})$. Since $f$ is an embedding, the image is contained in
$\widetilde{J^1}(M,\C^{n+1})$. We have the commutative diagram
\begin{equation*}
\xymatrix{
 & \widetilde{J^1}(M,\C^{n+1}) \ar[d]^{\pi_1} & \Sb' \ar@{_{(}->}[l]\\
M \ar[ur]^{j^1f} 
\ar[r]^{\hat f} & G_\R(n,\C^{n+1}) 
& \Sb \ar@{_{(}->}[l]\\
}
\end{equation*}
where $\hat f\colon M\to G_\R(n,\C^{n+1})$ is the mapping
$
x\mapsto T_{f(x)}f(M).
$
By definition of $\Sb$ we have that $T_{f(x)}f(M)$ is totally real if and only if 
$\hat f(x)\notin \Sb$. It follows that
\begin{equation}\label{trium}
\text{CRsing}\, f = (j^1f)^{-1}(\Sb').
\end{equation}


\begin{lemma}\label{blaa}
The mapping $\pi_1\colon\widetilde{J^1}(M,\C^{n+1})\to G_\R(n,\C^{n+1})$
is a submersion.
\end{lemma}

\begin{proof}
Let $(x_0,z_0,\tau_0)\in\widetilde{J^1}(M,\C^{n+1})$. We can choose real 
linear coordinates 
$\eta$ in $\C^{n+1}=\R^{2n+2}$ centered at $z_0$ such that
\begin{equation}\label{pulled}
\tau_0(T_{x_0}M) = \text{span}_\R\,
\begin{bmatrix}
I_n \\
0
\end{bmatrix}.
\end{equation}
Let $x$ be local coordinates for $M$ centered at $x_0$.  Then $\tau_0$ becomes  
a real $(2n+2)\times n$-matrix. In view of \eqref{pulled},
\begin{equation}\label{taag}
\tau_0=
\begin{bmatrix}
\tau_0' \\
0
\end{bmatrix},
\end{equation}
for some $\tau'_0\in GL(n,\R)$.
A neighborhood of $\tau_0$ in $\text{Hom}(\R^n,\R^{2n+2})$ is
parametrized by
$$
(D',D'')\in\R^{n\times n}\times\R^{(n+2)\times n}, \quad |D'|\ll 1, |D''|\ll 1,
$$
using
\begin{equation}\label{taag2}
(D',D'')\mapsto
\begin{bmatrix}
\tau_0' + D' \\
D''(\tau'_0+D')
\end{bmatrix}.
\end{equation}
It follows that $(x,\eta,D',D'')$ are local coordinates for $J^1(M,\C^{n+1})$ centered at 
$(x_0,z_0,\tau_0)$. Now,
\begin{equation}\label{bok}
\pi_1(x,\eta,D',D'')=
\text{span}_\R\,
\begin{bmatrix}
\tau_0' + D' \\
D''(\tau'_0+D')
\end{bmatrix}
=
\text{span}_\R\,
\begin{bmatrix}
I_n \\
D''
\end{bmatrix},
\end{equation}
cf.\ \eqref{ABA}. In view of \eqref{stad} thus 
$\pi_1(x,\eta,D',D'')=D''$ in local coordinates. Therefore
$\pi_1$ clearly is a submersion.
\end{proof}

By this lemma and Proposition~\ref{antarctica}, $\Sb_j'$ are submanifolds
of $J^1(M,\C^{n+1})$ of codimension $2j(1+j)$. It follows from
the jet transversality theorem, see, e.g., \cite[Theorem~3.2.9]{Hirsch}, that
if $f\in C^\infty(M,\C^{n+1})$, then 
\begin{equation}\label{jumpyard}
j^1f\pitchfork \Sb'_j, \quad j=1,\ldots, \lfloor n/2\rfloor,
\end{equation}
possibly after an arbitrarily small perturbation of
$f$. 
Since $\text{codim}\,\Sb'_j=2j(1+j)$, if $n\leq 11$ and
\eqref{jumpyard} hold, it follows that 
$j^1f(M)\cap \Sb'_j=\emptyset$ for $j\geq 2$ and 
\begin{equation}\label{natrium}
\text{CRsing}\, f= (j^1f)^{-1}(\Sb'_1)
\end{equation}
is either empty or a smooth submanifold of $M$ of codimension $4$.
By \eqref{trium} it is also closed.
This is essentially the proof of the following proposition.
We do not give more details since it is 
a special case of Theorem~\ref{sol} below.

\begin{proposition}\label{tornet}
Assume that $n\leq 11$ and let $f_0\colon M\to\C^{n+1}$ be an embedding.
In each neighborhood of $f_0$ in $C^\infty(M,\C^{n+1})$ there is an embedding
$f \colon M\to\C^{n+1}$ such that $\text{CRsing}\, f$ is either empty or a smooth closed submanifold of $M$ of codimension $4$.
Moreover, if $\tilde f$ is in a sufficiently small $C^2$-neighborhood of 
$f$ in $C^\infty(M,\C^{n+1})$, then $\tilde f$ is an embedding and
$\text{CRsing}\, \tilde f$ and 
$\text{CRsing}\, f$ are diffeomorphic.  
\end{proposition}

\section{Making $f(\text{CRsing}\,f)$ totally real}\label{crsingtotreal}
The first objective in this section is to find a set $\Sigma''\subset J^2(M,\C^{n+1})$ 
such that $j^2f(M)\cap\Sigma''=\emptyset$ if and only if
$f(\text{CRsing}\,f)$ is totally real. If $n\leq 5$, then $f(\text{CRsing}\,f)$
is trivially totally real so we here tacitly assume that $n\geq 6$.

Let $F_\Sb\to\Sb_1$ be the fiber bundle whose fiber over $P\in\Sb_1$ is
$G_\R(n-4,P)$.\footnote{Let $F$ be 
the flag manifold of real 
$(n-4,n)$-flags in $\C^{n+1}$; it is naturally 
a fiber bundle over $G_\R(n,\C^{n+1})$. The bundle $F_\Sb$ is the restriction of $F$
to $\Sb_1$.}
Let 
$$
\Sigma\subset F_\Sb
$$
be the set of $(P,\sigma)\in F_\Sb$ such that the real $n-4$-plane $\sigma$ 
contains a complex line. 

\begin{lemma}\label{spring}
The set $\Sigma$ is a smooth closed 
submanifold of $F_\Sb$ 
of codimension $8$. 
\end{lemma}

\begin{proof}
To see that $\Sigma$ is a submanifold,
let $P\in\Sb_1$ and recall from the proof of
Proposition~\ref{antarctica} that a neighborhood of $P$ in $\Sb_1$ can be parametrized 
by \eqref{pika}; $(u,v,\xi)\in\R^n\times\R^n\times\R^{(n+2)\times(n-2)}$ are 
local coordinates for $\Sb_1$ centered at $P$. Let $P_{u,v,\xi}\in\Sb_1$ 
be the real $n$-plane
given by \eqref{pika}.
Recall also from the proof of
Proposition~\ref{antarctica} that the $\R$-span of the 
first two columns of the right-hand side of \eqref{pika} is the unique complex line in
$P_{u,v,\xi}$. Let $\sigma$ be an 
$n-4$-plane in $P$ containing the complex line. For a suitable invertible real
$(n-2)\times(n-2)$-matrix $C$ then $\sigma$ is the $\R$-span of the first $n-4$ columns
of the $(2n+2)\times n$-matrix
$$
\begin{bmatrix}
I_2 & 0 \\
0 & C \\
0 & 0
\end{bmatrix}.
$$
It follows, cf.\ \eqref{stad}, that a chart in $F_\Sb$ centered at $(P,\sigma)$ is
given by
$$
(u,v,\xi,y)\mapsto
\left(
\text{span}_\R\,
\begin{bmatrix}
I_2 & \begin{matrix}
         \text{Re}\, h_{w,\xi} \\
         \text{Im}\, h_{w,\xi}
         \end{matrix} \\
\begin{matrix}
u & -v
\end{matrix} &
\begin{matrix}
I_{n-2} \\
\xi'
\end{matrix} \\
\begin{matrix}
v & u
\end{matrix} &
\xi''
\end{bmatrix}, 
\text{span}_\R\,
\begin{bmatrix}
I_2 & \begin{matrix}
         \text{Re}\, h_{w,\xi} \\
         \text{Im}\, h_{w,\xi}
         \end{matrix} \\
\begin{matrix}
u & -v
\end{matrix} &
\begin{matrix}
I_{n-2} \\
\xi'
\end{matrix} \\
\begin{matrix}
v & u
\end{matrix} &
\xi''
\end{bmatrix}
\begin{bmatrix}
I_2 & 0 \\
0 & C
\end{bmatrix}
\begin{bmatrix}
I_{n-4} \\
y
\end{bmatrix}
\right),
$$
where $y\in\R^{4\times (n-4)}$ and $w=u+iv\in\mathbb{C}^n$. 
In the local coordinates
$(u,v,\xi,y)$ for $F_\Sb$, the set $\Sigma$ is obtained by setting the first
two columns of $y$ to $0$. Thus,  $\Sigma$ is a submanifold of $F_\Sb$
of codimension $8$. 

To see that $\Sigma$ is closed, suppose that $(P_j,\sigma_j)\in\Sigma$ converges to
$(P,\sigma)\in F_\Sb$. 
As in Proposition~\ref{antarctica}, the set $S$ of points in $G_\R(n-4,\C^{n+1})$
corresponding to $n-4$-planes containing a complex line but no complex $2$-plane
is a submanifold. Moreover, $\overline{S}\setminus S$ is the set of $n-4$-planes
containing a complex $2$-plane. We have $\sigma_j\in S$ and $\sigma_j\to\sigma$
in $G_\R(n-4,\C^{n+1})$. Thus $\sigma\in \overline{S}$. But in fact, $\sigma\in S$
since $\sigma\subset P$
which contains a complex line but no complex $2$-plane.
Thus $(P,\sigma)\in\Sigma$, and so $\Sigma$ is closed.
\end{proof}

The reason for introducing $\Sigma$ is the following. Suppose that
$f\colon M\to\C^{n+1}$ is an embedding such that
 \eqref{jumpyard} holds.
Then, if $x\in \text{CRsing}\, f$ and $p=j^1f(x)$, 
$$
T_x \text{CRsing}\, f = (Dj^1f)_x^{-1} T_p\Sb'_1.
$$
Thus,
\begin{equation}\label{snart}
\big(T_{f(x)}f(M), \, T_{f(x)}f(\text{CRsing}\, f)\big)=
\Big(Df_x(T_xM),\, Df_x \big((Dj^1f)_x^{-1} T_p\Sb'_1\big)\Big)\in F_\Sb.
\end{equation}
By definition of $\Sigma$ it follows that
\begin{equation}\label{sno2}
T_{f(x)}f(\text{CRsing}\, f) \,\,\text{is totally real} \iff
\Big(Df_x(T_xM),\, Df_x \big((Dj^1f)_x^{-1} T_p\Sb'_1\big)\Big)\notin \Sigma.
\end{equation}

\smallskip

We will now define $\Sigma''$. To begin with, let 
$$
\Sb_1''\subset J^2(M,\C^{n+1})
$$ 
be the inverse image of $\Sb'_1$ under the natural submersion
$J^2(M,\C^{n+1})\to J^1(M,\C^{n+1})$.
Since $\Sb'_1$ is a submanifold of codimension $4$ also $\Sb''_1$ is.
Let $p\in\Sb''_1$. Recall that $s(p)=:x$ and $t(p)=:z$ are the source and target of 
$p$ respectively. Recall also that if $f\in C^\infty(M,\C^{n+1})$ is such that
$p=[f]_{x,2}$, then $D_p=Df_x$ and $H_p=Dj^1f_x$; cf.\ \eqref{Hp}.
Let 
$$\widetilde{\Sb''_1}\subset \Sb''_1
$$ 
be the set of $p\in\Sb''_1$
such that  $H_p\pitchfork T_{(s(p),t(p),D_p)}\Sb'_1$.
Clearly, $\widetilde{\Sb''_1}$ is open in $\Sb''_1$. It is also non-empty since if 
$g\in C^\infty(M,\C^{n+1})$ and $j^1g\pitchfork \Sb'_1$, 
then $[g]_{x,2}\in\widetilde{\Sb''_1}$ for any $x\in (j^1g)^{-1}\Sb'_1$. 
By jet transversality there are plenty of such $g$.
Let
\begin{equation}\label{sno}
\pi_2\colon \widetilde{\Sb''_1}\to F_\Sb,\quad
p\mapsto \Big(\pi_1\big(s(p),t(p),D_p\big),\, 
D_p \big(H_p^{-1} T_{(s(p),t(p),D_p)}\Sb'_1\big)\Big),
\end{equation}
where $\pi_1$ is the mapping \eqref{pinatrium}, and let
$$
\Sigma'':=\pi_2^{-1} \Sigma.
$$

\begin{proposition}\label{sommarlov}
Let $f\colon M\to\C^{n+1}$ be an embedding such that 
\eqref{jumpyard} holds. Then $f(\text{CRsing}\, f)$ it totally real if and only if
$j^2f(M)\cap\Sigma''=\emptyset$.
\end{proposition}

\begin{proof}
In view of \eqref{natrium} and the definition of $\Sb''_1$ we have
$(j^2f)^{-1}\Sb''_1=\text{CRsing}\, f$. Since $\Sigma''\subset\Sb''_1$,
what we need to show is 
thus that $f(\text{CRsing}\, f)$ it totally real if and only if
$j^2f(\text{CRsing}\, f)\cap\Sigma''=\emptyset$.

By \eqref{snart} and \eqref{sno} we have the commutative diagram
\begin{equation}\label{mixer}
\xymatrix{
 & \widetilde{\Sb''_1} \ar[d]^{\pi_2} & \Sigma'' \ar@{_{(}->}[l] \ar[d]\\
\text{CRsing}\, f \ar[ur]^{j^2f} 
\ar[r]^{\tilde f} & F_\Sb 
& \Sigma \ar@{_{(}->}[l]
}
\end{equation}
where $\tilde f$ is the mapping 
$$
x\mapsto \big(T_{f(x)}f(M), \, T_{f(x)}f(\text{CRsing}\, f)\big).
$$
By \eqref{snart} and \eqref{sno2},
$f(\text{CRsing}\, f)$ is totally real if and only if $\tilde{f}^{-1}\Sigma=\emptyset$.
Since $\tilde{f}^{-1}\Sigma=(j^2f)^{-1}\Sigma''$ by the commutativity of \eqref{mixer},
it follows that $f(\text{CRsing}\, f)$ is totally real if and only if
$j^2f(\text{CRsing}\, f)\cap\Sigma''=\emptyset$. This completes the proof.
\end{proof}

The next step is to show that $\Sigma''$ is a submanifold of $J^2(M,\C^{n+1})$
of codimension $12$. This follows from the next proposition
since $\Sigma''\subset\Sb''_1$, $\Sb''_1$ is a submanifold
of $J^2(M,\C^{n+1})$ of codimension $4$, and $\Sigma$ is a submanifold
of $F_\Sb$ of codimension $8$.

\begin{proposition}\label{submersion2}
The mapping $\pi_2\colon \widetilde{\Sb''_1}\to F_\Sb$ is a submersion.
\end{proposition}

We postpone the proof until the end of this section.


\begin{theorem}\label{sol}
Let $M$ be a compact smooth real $n$-dimensional manifold with $n\leq 11$
and let $f_0\colon M\to \C^{n+1}$ be an embedding. In each neighborhood of $f_0$
in $C^\infty(M,\C^{n+1})$ there is an embedding 
$f\colon M\to \C^{n+1}$ such that 
$\text{CRsing}\, f$ is either empty or a smooth closed submanifold of $M$ of codimension
$4$ and $f(\text{CRsing}\, f)$ is totally real. 
Moreover, if $\tilde f$ is in a sufficiently small $C^2$-neighborhood of $f$ in 
$C^\infty(M,\mathbb{C}^{n+1})$,
then $\tilde f$ is an embedding, $\text{CRsing}\, \tilde f$ and $\text{CRsing}\, f$ are diffeomorphic, and 
$\tilde f(\text{CRsing}\, \tilde f)$ is totally real.
%
\end{theorem}

\begin{proof}
By the jet transversality theorem the set of smooth mappings $f\colon M\to\C^{n+1}$
satisfying \eqref{jumpyard} and 
\begin{equation}\label{yardjump}
j^2 f \pitchfork\Sigma''
\end{equation}
is dense in $C^\infty(M,\C^{n+1})$. In each neighborhood of 
$f_0$ thus there is
such a smooth mapping.
Since $f_0$ is an embedding and the set of embeddings 
is open in $C^\infty(M,\C^{n+1})$, 
there is a smooth embedding  $f\colon M\to\mathbb{C}^{n+1}$
satisfying \eqref{jumpyard} and \eqref{yardjump} in each neighborhood of $f_0$.
Since $n\leq 11$ and $\text{codim}\, \Sigma''=12$, \eqref{yardjump}
precisely means that $j^2 f(M)\cap \Sigma''=\emptyset$. By 
Proposition~\ref{sommarlov} thus $f(\text{CRsing}\, f)$ is totally real.
By the paragraph preceding Proposition~\ref{tornet}, $\text{CRsing}\, f$ is 
either empty or a smooth closed submanifold of $M$ of codimension $4$.

It is clear that if $\tilde f\colon M\to\C^{n+1}$ is sufficiently $C^2$-close to $f$,
then $\tilde f$ is an embedding. To see that 
$\text{CRsing}\, \tilde f$ and $\text{CRsing}\, f$ are diffeomorphic and that
$\tilde f(\text{CRsing}\, \tilde f)$ is totally real
it suffices to see that $\tilde f$ satisfies \eqref{jumpyard} and \eqref{yardjump}.

Let $K:=J^1(M,\C^{n+1})\setminus\widetilde{J^1}(M,\C^{n+1})$; it is a closed set
and $j^1 f(M)\cap K=\emptyset$ since $ f$ is an embedding.
Let $S=\Sb_2\cup\cdots\cup\Sb_{\lfloor n/2\rfloor}$ and
$S'=\Sb'_2\cup\cdots\cup\Sb'_{\lfloor n/2\rfloor}$. 
By the paragraph preceding  Proposition~\ref{tornet} we have 
$j^1 f(M)\cap S'=\emptyset$ since \eqref{jumpyard} holds. Hence,
\begin{equation}\label{snabelulv}
j^1 f(M)\cap (K\cup S')=\emptyset.
\end{equation}
By Proposition~\ref{antarctica}, 
$S$ is closed, and since
$
S'
=\pi_1^{-1}S,
$
we see that $S'$ is closed in $\widetilde{J^1}(M,\C^{n+1})$. It follows that 
$K\cup S'$ is closed in $J^1(M,\C^{n+1})$.
Since $j^1 f(M)$ is compact it follows that $\tilde f$ 
satisfies \eqref{snabelulv} as well. 
Thus $j^1\tilde f(M)\cap\Sb'_j=\emptyset$, $j\geq 2$. Moreover,
since $j^1 f\pitchfork \Sb'_1$ and 
$\overline{\Sb'_1}\setminus \Sb'_1\subset K\cup S'$ in view of 
Proposition~\ref{antarctica},
it follows that $j^1 \tilde f\pitchfork \Sb'_1$. 
Hence, $\tilde f$ satisfies \eqref{jumpyard}.

It remains to see that 
 $j^2\tilde f\pitchfork \Sigma''$, which for dimensional reasons means that
$j^2\tilde f(M)\cap \Sigma''=\emptyset$.
If $j^1 f(M)\cap \Sb'_1=\emptyset$, then $\text{CRsing}\, f=\emptyset$.
It follows that $\text{CRsing}\, \tilde f=\emptyset$ and so, in particular,
 $j^2\tilde f(M)\cap \Sigma''=\emptyset$.
Now assume that $j^1 f(M)\cap \Sb'_1\neq\emptyset$. 
To see that $j^2\tilde f(M)\cap \Sigma''=\emptyset$ in this case
it suffices to show that $j^2 f(M)\cap\overline{\Sigma''}=\emptyset$. 
Since $j^1 f\pitchfork \Sb'_1$, for each $p\in j^2 f(M)\cap \Sb''_1$
we have that $H_p$, 
the differential of $j^1 f$ at $s(p)$, is transversal to 
$T_{(s(p),t(p),D_p)}\Sb'_1$. Thus 
\begin{equation}\label{snabelpenna}
j^2 f(M)\cap \Sb''_1\subset \widetilde{\Sb''_1}.
\end{equation}
Suppose now to get a contradiction that $j^2 f(M)\cap\overline{\Sigma''}\neq \emptyset$ and
let $p\in j^2 f(M)\cap\overline{\Sigma''}$.
By Lemma~\ref{spring} and Proposition~\ref{submersion2},
$\Sigma''$ is closed in $\widetilde{\Sb''_1}$. Since \eqref{yardjump} holds we have
$j^2 f(M)\cap\Sigma''=\emptyset$, and thus $p$ is in the closure of
$\widetilde{\Sb''_1}$ but not in $\widetilde{\Sb''_1}$.
By \eqref{snabelpenna}, if $p\in\Sb''_1$, then $p\in\widetilde{\Sb''_1}$, which 
is impossible. Thus, $p\in\overline{\Sb''_1}\setminus\Sb''_1$.
But then the image of $p$ under the natural submersion 
$J^2(M,\C^{n+1})\to J^1(M,\C^{n+1})$ is in $\overline{\Sb'_1}\setminus\Sb'_1$.
Since this image is $j^1 f(s(p))$ we get $j^1 f(s(p))\in K\cup S'$, which is a contradiction by \eqref{snabelulv}. 
Hence, $j^2 f(M)\cap\overline{\Sigma''}= \emptyset$ and the proof is complete.
\end{proof}

\begin{proof}[Proof of Proposition~\ref{submersion2}]
Let $p\in\widetilde{\Sb''_1}$. To begin with, we describe the tangent space of
$\Sb'_1$ at the image $(s(p),t(p),D_p)$ of $p$ in $J^1(M,\C^{n+1})$. We will
find $4$ linear functionals on $T_{(s(p),t(p),D_p)}J^1(M,\C^{n+1})$ such that the intersection of their kernels is $T_{(s(p),t(p),D_p)}\Sb'_1$.

Let $P=\pi_1(s(p),t(p),D_p)\in\Sb_1\subset G_\R(n,\C^{n+1})$.
We will use the $\C$-basis $\varepsilon$ for $\C^{n+1}$ introduced in
the proof of Proposition~\ref{antarctica}. As in that proof we will also
identify  $(z_0,\ldots,z_n)^t\in\C^{n+1}$
(with respect to the $\varepsilon$-basis) with real vectors \eqref{snorlaxen} 
in $\R^{2n+2}$. This gives us
real linear coordinates $\eta$ on $\C^{n+1}$ such that $P$ is given by the right-hand side of
\eqref{pulled}; cf.\ \eqref{pika} with $u=v=0=\xi$.
Following the proof of Lemma~\ref{blaa},
each choice of local coordinates $x$ for $M$ centered at $s(p)$ 
then gives us local coordinates
\begin{equation}\label{lokala}
(x,\eta,D',D'')\in\R^n\times \R^{2n+2}\times \R^{n\times n}\times\R^{(n+2)\times n},
\,\, |x|, |\eta|, |D'|, |D''|\ll 1,
\end{equation}
for $J^1(M,\C^{n+1})$ centered at $(s(p),t(p),D_p)$. Moreover,
$D''$ is local coordinates for $G_\R(n,\C^{n+1})$ centered at $P$ and
\begin{equation}\label{vrmote}
\pi_1(x,\eta,D',D'')=D''.
\end{equation}
With respect to these local coordinates, the tangent space of $\Sb_1$ at $P$ is parametrized
by \eqref{haalslag}. 

Let $\delta''$ be a tangent vector of $G_\R(n,\C^{n+1})$ at $P$; with respect to the 
local coordinates $D''$, $\delta''$ is
a real $(n+2)\times n$-matrix. We define functionals $\lambda_i$ on 
$T_PG_\R(n,\C^{n+1})$ by
$$
\lambda_1(\delta'')=Tr\left(
\begin{bmatrix}
1 & 0 & \cdots & 0 & 0 \\
0 & 0 & \cdots & -1 & 0 \\
\vdots & \vdots & & \vdots & \vdots \\
0 & 0 & \cdots & 0 & 0 
\end{bmatrix}
\delta''\right), \,\,
\lambda_2(\delta'')=Tr\left(
\begin{bmatrix}
0 & 1 & \cdots & 0 & 0 \\
0 & 0 & \cdots & 0 & -1 \\
\vdots & \vdots & & \vdots & \vdots \\
0 & 0 & \cdots & 0 & 0 
\end{bmatrix}
\delta''\right),
$$
$$
\lambda_3(\delta'')=Tr\left(
\begin{bmatrix}
0 & 0 & \cdots & 1 & 0 \\
1 & 0 & \cdots & 0 & 0 \\
\vdots & \vdots & & \vdots & \vdots \\
0 & 0 & \cdots & 0 & 0 
\end{bmatrix}
\delta''\right), \,\,
\lambda_4(\delta'')=Tr\left(
\begin{bmatrix}
0 & 0 & \cdots & 0 & 1 \\
0 & 1 & \cdots & 0 & 0 \\
\vdots & \vdots & & \vdots & \vdots \\
0 & 0 & \cdots & 0 & 0 
\end{bmatrix}
\delta''\right).
$$
In view of \eqref{haalslag}, $\delta''$
is in the tangent space of $\Sb_1$ at $P$ if and only if $\lambda_i(\delta'')=0$,
$i=1,\ldots,4$.
Identifying the local coordinates for $J^1(M,\C^{n+1})$ centered at
$(s(p),t(p),D_p)$ and tangent vectors at $(s(p),t(p),D_p)$,
by \eqref{vrmote} thus a tangent vector $(x,\eta,D',D'')$ at $(s(p),t(p),D_p)$ is 
in the tangent space of $\Sb'_1$ if and only if $\lambda_i(D'')=0$.
We notice that the functionals $\lambda_i$ depend on $P\in\Sb_1$. Since $\Sb_1$ is a submanifold, for $P'\in\Sb_1$ one can choose $\lambda_i^{P'}$ depending smoothly
on $P'\in\Sb_1$ such that $\delta''$ is in the tangent space of $\Sb_1$ at $P'$
if and only if $\lambda_i^{P'}(\delta'')=0$.

\smallskip

Next, we find an expression for the linear mapping $H_p$ from
$T_{s(p)}M$ to the tangent space of $J^1(M,\C^{n+1})$
at $(s(p),t(p),D_p)$. 
Let $g\in C^\infty(M,\C^{n+1})$ be a representative of $p$, and let $g^j$ be the components of $g$ with respect to the coordinates $\eta$ on $\R^{2n+2}=\C^{n+1}$. Consider the matrix-valued function
\begin{equation*}
\R^n\ni x\mapsto G(x) = \frac{\partial g^j}{\partial x_k}(x) \in \R^{(2n+2)\times n}.
\end{equation*}
Notice that by our choice of coordinates $\eta$, $G(0)$ equals
the right-hand side of
\eqref{taag} for some $\tau'_0\in GL(n,\R)$. After a linear change of the $x$-coordinates
we can assume that $\tau'_0=I_n$; this coordinate change does not affect the coordinates
$D''$ and, consequently, not the expression for the functionals $\lambda_i$.
A direct calculation shows that the differential of $G$ at $0$ is 
the linear mapping
$$
\R^n\ni x\mapsto
\begin{bmatrix}
\sum_\ell H^1_{1,\ell}x_\ell & \cdots & \sum_\ell H^1_{n,\ell}x_\ell \\
\vdots & & \vdots  \\
\sum_\ell H^{2n+2}_{1,\ell} x_\ell & \cdots & \sum_\ell H^{2n+2}_{n,\ell}x_\ell
\end{bmatrix}
\in \R^{(2n+2)\times n},
$$
where
$$
H^j_{k,\ell} = \frac{\partial^2 g^j}{\partial x_k\partial x_\ell}(0).
$$
The image of the differential thus is the $\R$-span of
the tangent vectors 
\begin{equation}\label{tangent}
\nu_1=
\begin{bmatrix}
H^1_{1,1} & \cdots & H^1_{n,1} \\
\vdots & & \vdots  \\
H^{2n+2}_{1,1} & \cdots & H^{2n+2}_{n,1}
\end{bmatrix},
\ldots,
\nu_n=
\begin{bmatrix}
H^1_{1,n} & \cdots & H^1_{n,n} \\
\vdots & & \vdots  \\
H^{2n+2}_{1,n} & \cdots & H^{2n+2}_{n,n}
\end{bmatrix}.
\end{equation}
The mapping \eqref{taag2} relates the non-standard coordinates $(D',D'')$ for 
$\R^{(2n+2)\times n}$ centered at $0$ to the standard ones
centered at $G(0)$; recall that $G(0)$ equals
the right-hand side of
\eqref{taag} with $\tau'_0=I_n$.
Under this coordinate change, a tangent vector $(\delta',\delta'')$ at $0$ transforms as
$$
(\delta',\delta'')\mapsto 
\begin{bmatrix}
\delta' \\
\delta''
\end{bmatrix}.
$$
Under the inverse of \eqref{taag2} thus the tangent vectors \eqref{tangent} 
transform as
\begin{equation}\label{bil}
\nu_\ell\mapsto (\nu'_\ell,\nu''_\ell),\quad
\nu'_\ell=
\begin{bmatrix}
H^1_{1,\ell} & \cdots & H^1_{n,\ell} \\
\vdots & & \vdots  \\
H^{n}_{1,\ell} & \cdots & H^{n}_{n,\ell}
\end{bmatrix}, \,\,
\nu''_\ell=
\begin{bmatrix}
H^{n+1}_{1,\ell} & \cdots & H^{n+1}_{n,\ell} \\
\vdots & & \vdots  \\
H^{2n+2}_{1,\ell} & \cdots & H^{2n+2}_{n,\ell}
\end{bmatrix}.
\end{equation}
With respect to the local coordinates $(x,\eta,D',D'')$ for $J^1(M,\C^{n+1})$
centered at $(s(p),t(p),D_p)$, 
we get that $H_p$ is the linear mapping
\begin{equation}\label{framme}
H_p\colon x\mapsto \big(x,G(0)x,\sum_\ell \nu'_\ell x_\ell, \sum_\ell \nu''_\ell x_\ell\big),
\end{equation}
where we have identified the local coordinates $x$ centered at $s(p)$ with 
tangent vectors at $s(p)$.

\smallskip

We now check that $\pi_2$ is smooth. 
The local coordinates \eqref{lokala}
for $J^1(M,\C^{n+1})$ can be extended to local coordinates 
$$
(x,\eta,D',D'',h^1,\ldots,h^{2n+2})\in  
\R^n\times \R^{2n+2}\times \R^{n\times n}\times\R^{(n+2)\times n}
\times Sym(n)^{2n+2}
$$
for $J^2(M,\C^{n+1})$ centered at $p$. Here $h^j$ are symmetric real 
$n\times n$-matrices such that 
$
|h^j|\ll 1
$.
If $p'=(x,\eta,D',D'',h)$, then
$H_{p'}$ is given by  \eqref{framme}
with $\nu'_\ell$ and $\nu''_\ell$ replaced by $\nu'_\ell(h)$ and $\nu''_\ell(h)$, respectively,
where $\nu'_\ell(h)$ and $\nu''_\ell(h)$ are as in \eqref{bil} with
$H^j_{k\ell}$ replaced by $H^j_{k\ell}+h^j_{k\ell}$.
Assume that $P'=\pi_1(s(p'),t(p'),D_p)\in\Sb_1$ and let
$$
\alpha(h)=
\begin{bmatrix}
\lambda_1(\nu''_1(h)) & \cdots & \lambda_1(\nu''_n(h)) \\
\vdots & & \vdots \\
\lambda_4(\nu''_1(h)) & \cdots & \lambda_4(\nu''_n(h))
\end{bmatrix},
$$
where $\lambda_i=\lambda_i^{P'}$. 
Notice that $\alpha$ depends smoothly on $h$ and $P'$. 

In view of \eqref{framme} and the first part of the proof, 
$H_{p'}(x)$ is in the tangent space of $\Sb'_1$ at 
$(s(p'),t(p'),D_{p'})$ if and only if $\alpha(h) x=0$.
Since $H_{p'}\pitchfork T_{(s(p'),t(p'),D_{p'})}\Sb'_1$, the kernel of $\alpha$
has codimension $4$. Thus $\alpha$ has an invertible $4\times 4$-minor. 
For a suitable invertible $n\times n$-matrix $A$ thus
$$
\alpha(h)A=
\begin{bmatrix}
\alpha'(h) & \alpha''(h)
\end{bmatrix}
=
\begin{bmatrix}
0 & I_4
\end{bmatrix}
+O(|h|+|D'|+|D''|).
$$
After the linear change $\tilde x =A^{-1}x$ of the coordinates in $M$, 
but keeping the coordinates $(x,\eta,D',D'',h)$ for $J^2(M,\C^{n+1})$,
we then have  
$\alpha(h)A\tilde x=0$ if and only if 
$$
x''=-\big(\alpha''(h)\big)^{-1}\alpha'(h)x',
$$
where $\tilde x=(x',x'')^t$, $x'\in\R^{n-4}$, $x''\in\R^4$. It follows that
\begin{equation}\label{sudd}
D_{p'}\big(H^{-1}_{p'}T_{(s(p'),t(p'),D_p)}\Sb'_1\big)
=\text{span}_\R\,
\begin{bmatrix}
I_n+D' \\
D''(I_n+D')
\end{bmatrix}
\begin{bmatrix}
-I_{n-4} \\
\big(\alpha''(h)\big)^{-1}\alpha'(h)
\end{bmatrix}.
\end{equation}
The right-hand side depends smoothly on $D'$, $D''$, and $h$. By 
\eqref{sno} thus $\pi_2$ is smooth.

\smallskip

Finally we show that the differential of $\pi_2$ at $p$ is surjective. 
In view of \eqref{sno} and Lemma~\ref{blaa} it suffices to see that \eqref{sudd} 
as a function of $h$, with $D'=0=D''$, to $G_\R(n-4,P)$ has surjective differential
at $h=0$. In Grassmannian coordinates, cf.\ \eqref{stad}, this amounts to show that
\begin{equation}\label{dyk}
h\mapsto -\big(\alpha''(h)\big)^{-1}\alpha'(h)
\end{equation}
has surjective differential at $h=0$. We have that
$\alpha'(h)$ and  $\alpha''(h)$ are linear in $h$. Since $\alpha'(0)=0$ and 
$\alpha''(0)=I_4$ a straightforward calculation gives that the differential of
\eqref{dyk} at $h=0$ is the linear mapping
$$
Sym(n)^{2n+2}\ni \mathfrak{h}\mapsto -\alpha'(\mathfrak{h})
=
-\begin{bmatrix}
\lambda_1(\nu''_1(\mathfrak{h})) & \cdots & \lambda_1(\nu''_{n}(\mathfrak{h})) \\
\vdots & & \vdots \\
\lambda_4(\nu''_1(\mathfrak{h})) & \cdots & \lambda_4(\nu''_{n}(\mathfrak{h}))
\end{bmatrix}A',
$$
where $A'$ is the first $n-4$ columns of $A$ and $\nu''_\ell(\mathfrak{h})$ are as in
\eqref{bil} with $H^j_{k,\ell}$ replaced by $\mathfrak{h}^j_{k,\ell}$. We need to check
that $\alpha'(\mathfrak{h})$  can be any given $4\times (n-4)$-matrix by choosing
$\mathfrak{h}$ appropriately. Since $A'$ contains an invertible $(n-4)\times(n-4)$-minor
it suffices to see that
$\alpha(\mathfrak{h})$
can be any given $4\times n$-matrix by choosing
$\mathfrak{h}$ appropriately. A simple calculation gives
$$
\lambda_1(\nu''_\ell(\mathfrak{h})) = 
\mathfrak{h}^{n+1}_{1,\ell}-\mathfrak{h}^{2n+1}_{2,\ell}, \quad
\lambda_2(\nu''_\ell(\mathfrak{h})) = 
\mathfrak{h}^{n+2}_{1,\ell}-\mathfrak{h}^{2n+2}_{2,\ell},
$$
$$
\lambda_3(\nu''_\ell(\mathfrak{h})) = 
\mathfrak{h}^{n+1}_{2,\ell} + \mathfrak{h}^{2n+1}_{1,\ell}, \quad
\lambda_4(\nu''_\ell(\mathfrak{h})) = 
\mathfrak{h}^{n+2}_{2,\ell} + \mathfrak{h}^{2n+2}_{1,\ell}.
$$
It is then straightforward to check that one can choose the symmetric matrices
$\mathfrak{h}^{n+1}$, $\mathfrak{h}^{n+2}$, $\mathfrak{h}^{2n+1}$, and
$\mathfrak{h}^{2n+2}$ so that entry $(i,\ell)$ in $\alpha(\mathfrak{h})$ becomes $1$ 
and all other entries $0$. For instance, one obtains this for entry $(1,\ell)$ with $\ell\geq 3$
if $\mathfrak{h}^j=0$ for all $j\neq n+1$ and
$\mathfrak{h}^{n+1}$ has $1$ in entries $(1,\ell)$
and $(\ell,1)$ and $0$ elsewhere. Since $\alpha(\mathfrak{h})$ is linear it follows that 
it can be any given $4\times n$-matrix by choosing $\mathfrak{h}$ appropriately.
Hence, $\pi_2$ has surjective differential at $p$, and the proof is complete.
\end{proof}

%
%
%
%

\section{Proof of Theorem~\ref{main1}}\label{pfthm}

Let $f\colon M\to \C^{n+1}$ be an embedding as in Theorem~\ref{sol}.
Assume first that $\text{CRsing}\, f$ is empty. Then by \cite[Theorem~2]{LW}
there is a smooth perturbation $\varphi\colon f(M)\to\mathbb{C}^{n+1}$
such that $\varphi\circ f(M)$ is polynomially convex and totally real.
Thus $\varphi\circ f\colon M\to\mathbb{C}^{n+1}$ is an embedding with the
required properties.

Assume now instead that $\text{CRsing}\, f$ is a closed submanifold and that
$f(\text{CRsing}\, f)$ is totally real. In this case, to conclude the proof of Theorem~\ref{main1}, we will use the following theorem. As mentioned,
it is a slight generalization
of \cite[Theorem~1.4]{AW} and is proved in \cite{ASKW}.

\begin{theorem}\label{trocadero}
Let $K\subset X\subset \C^m$ be compact sets such that $K$ is polynomially convex and
$X\setminus K$ is a totally real manifold of (real) dimension $d<m$.
For any $k\in\mathbb{N}$ and any neighborhood $\mathcal{U}$ of $\text{id}_{\C^m}$
in $C^k$-topology,
there is a $C^\infty$ diffeomorphism $\psi\colon \C^m\to\C^m$ such that
\begin{itemize}
\item[(i)] $\psi\in\mathcal{U}$ and $\psi|_K=\text{id}_K$,
\item[(ii)]  for each $z\in X$, the differential $D\psi_z\colon \C^m\to\C^m$
is $\C$-linear,
\item[(iii)] $\psi(X)$ is polynomially convex.
\end{itemize}
\end{theorem}

We apply Theorem~\ref{trocadero} with $X=f(\text{CRsing}\, f)$ and $K=\emptyset$.
We get a diffeomorphism $\psi_1\colon\C^{n+1}\to\C^{n+1}$ such that 
$\psi_1\circ f(\text{CRsing}\, f)$ is polynomially convex.  Let $f_1=\psi_1\circ f$;
it is an embedding $M\to\mathbb{C}^{m+1}$ since $\psi_1$ is a diffeomorphism.
Notice that in addition to being polynomially convex, $f_1(\text{CRsing}\, f)$ is
totally real since  $\psi_1$ is a diffeomorphism satisfying (ii).
We claim that
\begin{equation}\label{midsummer}
\text{CRsing}\, f_1 = \text{CRsing}\, f
\end{equation}
if $\psi_1$ is sufficiently $C^2$-close to $\text{id}_{\C^{n+1}}$.
Since $D\psi_1$ is $\C$-linear on $f(\text{CRsing}\, f)$ it follows that
$\text{CRsing}\, f_1 \supset \text{CRsing}\, f$. The inclusion is open since
these manifolds have the same dimension. Since  $\text{CRsing}\, f$ is closed thus 
$\text{CRsing}\, f$ 
is a union of connected components of $\text{CRsing}\, f_1$. The claim thus follows since
 $\text{CRsing}\, f_1$ and $\text{CRsing}\, f$
are diffeomorphic by Theorem~\ref{sol}.
Now since $f_1(\text{CRsing}\, f)$ is polynomially convex and totally real,
by \eqref{midsummer} also $f_1(\text{CRsing}\, f_1)$ is. Hence,
$f_1$ satisfies (b) of Theorem~\ref{main1}.


We now apply Theorem~\ref{trocadero} with $X=f_1(M)$ and 
$K=f_1(\text{CRsing}\, f_1)$. We get a diffeomorphism 
$\psi_2\colon\C^{n+1}\to\C^{n+1}$ such that, letting $f_2=\psi_2\circ f_1$, 
$f_2(M)$ is 
polynomially convex and 
$f_2(\text{CRsing}\, f_1)=f_1(\text{CRsing}\, f_1)$. 
Since $\psi_2$ is a diffeomorphism satisfying (ii) on $f_1(M)$ 
it follows that $f_2(\text{CRsing}\, f_2)=f_1(\text{CRsing}\, f_1)$.
Thus, $f_2\colon M\to\C^{n+1}$ is an embedding
satisfying (a) and (b) of Theorem~\ref{main1}, and so Theorem~\ref{main1} is proved.

\section{Proof of Theorem~\ref{main3}}\label{extra}
Let now $\Sb_j\subset G_\R(n,\C^{m})$ be the set of 
real $n$-planes in $\C^{m}$ containing a complex $j$-plane but no complex
$j+1$-plane. As in the proof of Proposition~\ref{antarctica} one can show that
$\Sb:=\cup_j\Sb_j$ is a real-analytic closed connected subset of $G_\R(n,\C^{m})$, 
$\Sb_j$ is a submanifold of $G_\R(n,\C^{m})$ of codimension $2j(m-n+j)$, and 
$\overline{\Sb}_j\setminus \Sb_j=\cup_{k>j}\Sb_k$.

Let $\pi_1$ be the mapping  \eqref{pinatrium} with $\C^{n+1}$ replaced by $\C^m$.
It is shown to be a  submersion as in the proof of Lemma~\ref{blaa}. Thus
$\Sb'_j:=\pi_1^{-1}(\Sb_j)$ is a submanifold of $J^1(M,\C^m)$ of codimension
$2j(m-n+j)$. It follows that if $n\geq 12$ and $m=\lfloor 5n/4\rfloor -1$, then
$$
n<\text{codim}\, \Sb'_j,\qquad j\geq 2.
$$
If $f\colon M\to\C^m$ thus $j^1f(M)\cap \Sb'_j=\emptyset$, $j\geq 2$, possibly after 
an arbitrarily small perturbation of $f$. Hence,
$$
\text{CRsing}\, f=(j^1f)^{-1} (\cup_j\Sb'_j)=(j^1f)^{-1}(\Sb'_1)
$$
is either empty or a submanifold of $M$ of codimension $2(m-n+1)$, possibly after 
an arbitrarily small perturbation of $f$.
Notice that if $n-2(m-n+1)<2$, then $\text{dim}\, \text{CRsing}\, f<2$
and thus $f(\text{CRsing}\, f)$ is totally real. 
In what follows we tacitly assume that $n-2(m-n+1)\geq 2$.

\smallskip

Let $F_\Sb\to\Sb_1$ be the fiber bundle whose fiber over $P\in\Sb_1$
is $G_\R(n-2(m-n+1),P)$ and let $\Sigma\subset F_\Sb$ be the set of flags
$(P,\sigma)\in F_\Sb$ such that the subspace $\sigma$ of $P$ contains a complex line.
As in Lemma~\ref{spring} one shows that $\Sigma$ is a closed submanifold of
$F_\Sb$ of codimension $4(m-n+1)$. Let us do the calculation of the codimension in
some detail. If  $P\in\Sb_1$, then $P=\ell\oplus \Pi$, where $\ell$ is a complex line
in $\C^m$ and $\Pi$ is a real $n-2$-plane in $\ell^\perp\simeq \C^{m-1}$. A flag
$(P,\sigma)$ is in $\Sigma$ if and only if $\sigma=\ell\oplus\pi$, where $\pi$
is a real linear subspace of $\Pi$ of dimension $n-2(m-n+1)-2$. Hence,
$$
\text{dim}\, \Sigma=\text{dim}\, \Sb_1+\text{dim}\, G_\R(n-2(m-n+1)-2,\Pi).
$$
A simple calculation using \eqref{halt} then shows that 
$\text{codim}\, \Sigma=4(m-n+1)$.

Let $\Sb''_1\subset J^2(M,\C^m)$ be the inverse image of $\Sb'_1$ under the natural submersion $J^2(M,\C^m)\to J^1(M,\C^m)$. We define $\pi_2$ as in 
\eqref{sno} and let $\Sigma''=\pi_2^{-1}\Sigma$. As in the proof of 
Proposition~\ref{sommarlov} one shows that if $f\colon M\to\C^m$ is an embedding
such that \eqref{jumpyard} holds,
then $f(\text{CRsing}\, f)$ it totally real if and only if
$j^2f(M)\cap\Sigma''=\emptyset$. The mapping $\pi_2$ is shown to be a submersion
as in the proof of Proposition~\ref{submersion2}. Thus $\Sigma''$ is a submanifold
of $J^2(M,\C^m)$ and
\begin{eqnarray*}
\text{codim}\,\Sigma'' &=&
\text{codim}\, \Sb''_1+\text{codim}_{\Sb''_1}\Sigma''
=
\text{codim}\, \Sb_1+\text{codim}\,\Sigma \\
&=&
2(m-n+1)+4(m-n+1)
=
6(m-n+1).
\end{eqnarray*}
If $n\geq 12$ and $m=\lfloor 5n/4\rfloor-1$, then $n<6(m-n+1)$.
It thus follows that $j^2f(M)\cap\Sigma''=\emptyset$, possibly after a perturbation of
the embedding $f\colon M\to\C^m$, as long as $n\geq 12$ and $m=\lfloor 5n/4\rfloor-1$.
In the same way as Theorem~\ref{sol} is proved one now obtains

\medskip

\noindent {\bf Theorem~\ref{sol}'.}
\emph{Let $M$ be a compact smooth real $n$-dimensional manifold with $n\geq 12$,
let $m=\lfloor 5n/4\rfloor-1$,
and let $f_0\colon M\to \C^{m}$ be an embedding. In each neighborhood of $f_0$
in $C^\infty(M,\C^{m})$ there is an embedding 
$f\colon M\to \C^{m}$ such that 
$\text{CRsing}\, f$ is either empty or a smooth closed submanifold of $M$ of codimension
$2(m-n+1)$ and $f(\text{CRsing}\, f)$ is totally real. 
Moreover, if $\tilde f$ is in a sufficiently small $C^2$-neighborhood of $f$ in 
$C^\infty(M,\mathbb{C}^{m})$,
then $\tilde f$ is an embedding, $\text{CRsing}\, \tilde f$ and $\text{CRsing}\, f$ are diffeomorphic, and 
$\tilde f(\text{CRsing}\, \tilde f)$ is totally real.}

\medskip

Using this and Theorem~\ref{trocadero} one concludes Theorem~\ref{main3}
as in Section~\ref{pfthm}.

\end{document}